\documentclass[a4paper,12pt]{amsart}
\usepackage{geometry}           
\usepackage{graphicx}	
\usepackage{amsthm}
\usepackage{amssymb}
\usepackage{amsfonts}
\usepackage{enumerate}
\usepackage{mathrsfs}
\usepackage[inline]{enumitem}
\usepackage{color}

\newtheorem{thm}{Theorem}[section]

\newtheorem{prop}[thm]{Proposition}
\newtheorem{lem}[thm]{Lemma}
\newtheorem{cor}[thm]{Corollary}
\newtheorem{examp}[thm]{Example}

\theoremstyle{definition}
\newtheorem{defn}[thm]{Definition}
\newtheorem{qu}[thm]{Question}
\newtheorem{rmk}[thm]{Remark}

\newcommand{\AU}{\mathrm{AU}}
\newcommand{\Z}{\mathrm{Z}}

\newcommand{\bN}{\mathbb{N}}
\newcommand{\bZ}{\mathbb{Z}}

\newcommand{\defbold}{\textbf}

\newcommand{\mc}{\mathcal}

\def\part{\mathcal{P}}
\def\qart{\mathcal{Q}}

\begin{document}

\title[Topologically simple t.d.l.c. matrix groups]{Topologically simple, totally disconnected, locally compact infinite matrix groups}

\author{P.\,Groenhout, C.\,D.~Reid and G.\,A.~Willis}
\address{The University of Newcastle, Callaghan 2308, N.S.W. Australia}
\thanks{This research was supported by A.R.C. Grant FL170100032}

\date{}

\maketitle

\begin{abstract}
Groups of almost upper triangular infinite matrices with entries indexed by integers are studied. It is shown that, when the matrices are over a finite field, these groups admit a nondiscrete totally disconnected, locally compact group topology and are topologically simple.
\end{abstract}

\section{Introduction}

The present article contributes to the theory of totally disconnected, locally compact (t.d.l.c.) groups by constructing examples of simple t.d.l.c. groups.  To explain their significance, we first recall some of the theoretical context.

A key aspect of the theory of locally compact groups is the relationship between the local structure of a group $G$, that is, the properties shared by all the neighbourhoods of the identity in $G$, and the global properties of $G$.  For connected locally compact groups, the connection is strong and well-understood: every such group is a pro-Lie group, and the structure of connected pro-Lie groups is controlled to a great extent by the associated pro-Lie algebra, which precisely captures the local structure (see \cite{HofmannMorris}).  In the complementary case of totally disconnected, locally compact (t.d.l.c.) groups, the local structure is given by Van Dantzig's theorem: there is a base of neighbourhoods of the identity consisting of compact open subgroups.  Since every open subgroup of a compact group has finite index, the local structure therefore consists of the properties of a compact open subgroup that are invariant on passage to a subgroup of finite index.  One can derive an exact analogue of Lie theory for analytic groups defined over totally disconnected locally compact fields such as $\mathbb{Q}_p$ and $\mathbb{F}_q(\!(t)\!)$.  However, the class of t.d.l.c.~groups is much larger than just the analytic case, and compared to (pro-)Lie groups, the general connection between local and global structure is less rigid and much less well-understood.  There remains a fundamental and largely unanswered question of which local structures are compatible with which global properties.  For example:

\begin{qu}\label{qu:local_one}Which compact groups can occur (up to finite index) as open subgroups of (topologically) simple groups?\end{qu}

The following is a local condition that ensures a strong interaction between the local and global structure: suppose that $G$ is nondiscrete, but the quasi-centre $\mathrm{QZ}(G)$, that is, the set of elements with open centraliser, is discrete. In this situation, the quotient $G/\mathrm{QZ}(G)$ is locally isomorphic to $G$ and has trivial quasi-centre, so one effectively reduces to the case where the quasi-centre is trivial.  Here we recall the framework introduced in \cite{CapDeMedts} and \cite{BEW}.  Given a group $G$ with trivial quasi-centre, the group of germs $\mathscr{L}(G)$ consists of all isomorphisms between open subgroups of $G$, modulo equality on an open set.  In this context, Question~\ref{qu:local_one} is subject to the following dichotomy: either there are no (topologically) simple groups locally isomorphic to $G$, or else the group of germs $\mathscr{L}(G)$ has the property that the subgroup $R := \mathrm{Res}(\mathscr{L}(G))$, defined as the intersection of open normal subgroups, is open and (topologically) simple.   In the latter case, $R$ is the unique largest (topologically) simple group of this local isomorphism type.  So for profinite groups with trivial quasi-centre, the previous question can be restated as follows:

\begin{qu}Given a profinite group $U$ with trivial quasi-centre, when is $R = \mathrm{Res}(\mathscr{L}(U))$ open and simple?\end{qu}

A more specific question is to give a local description of the class $\mathscr{S}$ of nondiscrete compactly generated topologically simple t.d.l.c.~groups $G$.  Such groups play an important role in the general theory \cite{BurgerMozes}, \cite{CapMon}, and significant progress has been made in recent years both in obtaining restrictions on the possible local structure \cite{CapReidWiII} and constructing new examples; for instance it was shown by Smith \cite{SMSmith} that there are $2^{\aleph_0}$ isomorphism classes in $\mathscr{S}$, but it is not known if there are uncountably many local isomorphism classes in $\mathscr{S}$.  The condition of compact generation imposes additional restrictions on the local structure.  For example, every group in $\mathscr{S}$ has trivial quasi-centre (\cite[Theorem~4.8]{BEW}) and no nontrivial abelian subgroup of a group in $\mathscr{S}$ has open normaliser (\cite[Theorem~A]{CapReidWiII}); neither of these statements is true if one drops the compact generation requirement.

Examples are known of nondiscrete topologically simple t.d.l.c. groups with dense quasi-centre, for example in~\cite{WSimp}.  Other examples of non-compactly generated simple groups have been obtained where the local structure is an iterated wreath product of finite groups, and where the quasi-centre is trivial (\cite[Lemma~6.9]{CapDeMedts}, \cite[Example~6.3(v)]{ReidJI}).  Such iterated wreath product constructions however have a local structure similar to that found in the known examples of groups in $\mathscr{S}$.  

In the present article, infinite-dimensional matrix groups over finite fields are constructed and shown to non-compactly generated and topologically simple.  Their topological simplicity is based on the simplicity of the finite matrix groups $\mbox{PSL}_n(\mathbb{F}_q)$ for~$n\geq2$.  Like the examples mentioned in the previous paragraph, the examples are direct limits of profinite groups.  The examples have trivial quasi-centre, so they fit into the framework of groups of germs, but nevertheless they have a fundamentally different local structure from groups in $\mathscr{S}$, because they have nontrivial abelian subgroups with open normaliser.

The topologically simple groups $\AU_{\Lambda}(\mathbb{F}_q)/\Z_{\Lambda}(\mathbb{F}_q)$ constructed in this article all have compact open subgroups $\mathrm{U}_{\Lambda}(\mathbb{F}_q)/\Z_{\Lambda}(\mathbb{F}_q)$ formed as upper-triangular matrices modulo scalar matrices, with respect to a preorder $\Lambda$ on the coordinates satisfying certain conditions.  The general theory then implies that there is a unique largest topologically simple group
\[
\mathrm{RU}_{\Lambda}(\mathbb{F}_q) := \mathrm{Res}(\mathscr{L}(\mathrm{U}_{\Lambda}(\mathbb{F}_q)/\Z_{\Lambda}(\mathbb{F}_q)))
\]
of the same local isomorphism type, which cannot be compactly generated.

\begin{qu}For which preorders $\Lambda$ and finite fields $\mathbb{F}_q$ is $\AU_{\Lambda}(\mathbb{F}_q)/\Z_{\Lambda}(\mathbb{F}_q) = \mathrm{RU}_{\Lambda}(\mathbb{F}_q)$?  When they are not equal, what is $\mathrm{RU}_{\Lambda}(\mathbb{F}_q)$?
\end{qu}

The article is structured as follows: The matrix groups are defined in \S2 and are indexed by preordered sets satisfying certain properties. In \S3 algebraic properties of these groups are developed and subgroups important for later results are identified. The topology on the groups is described in \S4. Topological simplicity of the groups is established in \S5 and is shown in \S6 that there are uncountably many different local isomorphism types of matrix groups of this kind.

\begin{rmk}
The groups studied here have been described previously. In the case when the pre-order is $\mathbb{N}$ with its usual order, they were introduced and their representation theory investigated by A. M. Vershik and A. Zelevinsky in the early 1980s and by S. V. Kerov and A. M. Vershik in the 1990s, see~\cite{Kerov-Vershik}. The current state of the representation theory is given in~\cite{Gorin-Kerov-Vershik}, which also motivates study of the group and gives a short history of its representation theory with many more references. (The group called $\AU_{\mathbb{N}}(\mathbb{F}_q)$ in the present paper is called $\mathbb{GLB}$ in~\cite{Gorin-Kerov-Vershik}.) Local compactness of $\mathbb{GLB}$ and Haar measure are important for the representation theory.

Parabolic subgroups and the commutator subgroup of $\mathbb{GLB}$, and of similarly defined groups over more general rings, are described in \cite{Gupta-Holub1, Gupta-Holub2, Holub}. It is noted in those papers that $\mathbb{GLB}$ is topologically simple modulo its centre. The novelty of the present paper is that we define uncountably many non-isomorphic infinite-matrix groups over each finite field and study their local structure. 

We are grateful to W.~Ho\l ubowski for drawing our attention to the previous work on these groups after the first version of our paper was posted on the arXiv.
\end{rmk}

\section{Definitions}
\label{sec:definitions}

The groups of interest are infinite matrix groups in which the matrices are indexed by the partially ordered set $(\mathbb{Z},\leq)$ or $(\mathbb{N},\leq)$. Since the construction works in greater generality, we begin by abstracting the properties which the indexing set needs to have for the arguments to go through. 
\begin{defn}
\label{defn:Z-like}
Fix a preordered set $(\Lambda,\lesssim)$, so that $\lesssim$ is reflexive and transitive. Then: `$i\lesssim j$ and $j\lesssim i$' is abbreviated to $i\sim j$; `$i\lesssim j$ and $i\not\sim j$' to $i\lnsim j$; and the interval notation $[i,j]$ refers to $\{k\in\Lambda \mid i\lesssim k\lesssim j\}$.  The subset $\Lambda'$ of $\Lambda$ is \defbold{convex} if, whenever $i,j \in \Lambda'$, then $[i,j] \subseteq \Lambda'$; and $\Lambda'$ is \defbold{strongly convex} if for all $i \in \Lambda \setminus \Lambda'$, either $i \lnsim j$ for all $j \in \Lambda'$, or else $i \gnsim j$ for all $j \in \Lambda'$.  Say that $(\Lambda,\lesssim)$ is \defbold{$\bZ$-like} if every finite subset of $\Lambda$ is contained in a finite strongly convex subset of $\Lambda$.
\end{defn}
\noindent Note that, as the name implies, every strongly convex set is convex.  If $(\Lambda,\lesssim)$ is $\bZ$-like, then $[i,i] = \{j\in\Lambda \mid j\sim i\}$, the convex hull of $\{i\}$, is contained in a strongly convex subset of~$\Lambda$ for every $i\in\Lambda$ and hence it is finite. In the cases when $\Lambda$ is $\mathbb{N}$, $-\mathbb{N}$ or $\mathbb{Z}$ with their usual ordering, a finite strongly convex subset is just an interval $[m,n]$ with $m<n$. The following observation will be useful.
\begin{lem}
\label{lem:str_conv}
The intersection of strongly convex sets is strongly convex. Hence every non-empty subset of $\Lambda$ is contained in a smallest strongly convex set.
\end{lem}
\begin{proof}
Let $\{\Lambda_\alpha\}$ be a set of strongly convex sets and suppose that $x\not\in \bigcap_\alpha \Lambda_\alpha$. Then there is $\alpha$ such that $x\not\in \Lambda_\alpha$ and so either $x\lnsim y$ for all $y\in \Lambda_\alpha$ or $y\lnsim x$ for all $y\in\Lambda_\alpha$. Hence either $x\lnsim y$ for all $y\in \bigcap_\alpha \Lambda_\alpha$ or $y\lnsim x$ for all $y\in\bigcap_\alpha \Lambda_\alpha$ and $\bigcap_\alpha \Lambda_\alpha$ is strongly convex.
\end{proof}

Other examples of $\bZ$-like partially ordered sets and a construction on such sets which will be used to construct examples of groups are described next.
\begin{prop}
\label{prop:Z-like}
\begin{enumerate}[label = (\roman*)]
\item Let $\{(\Lambda_n,\lesssim_n)\}$ be a set of preordered finite sets, indexed by $\mathbb{N}$, $-\mathbb{N}$ or $\mathbb{Z}$ and define, for $x,y\in \bigsqcup_n \Lambda_n$,
$$
i\lesssim j \mbox{ holds if } \begin{cases}
\mbox{ either }i\in \Lambda_m \mbox{ and } j\in \Lambda_n \mbox{ with }m<n, \\
\mbox{ or } i,j\in \Lambda_n\mbox{ for some }n\mbox{ and }i\lesssim_n j
\end{cases}.
$$
Then $(\bigsqcup_n \Lambda_n,\lesssim)$ is a $\mathbb{Z}$-like partially ordered set. The subsets $\Lambda_m\sqcup \dots \sqcup \Lambda_n$, with $m\leq n$, are strongly convex. 
\label{prop:Z-like1}
\item Suppose that $(\Lambda,\lesssim)$ is $\mathbb{Z}$-like and that $\part = \left\{ \mathfrak{p}_\alpha\right\}$ is a partition of~$\Lambda$ into finite convex subsets, all but finitely many of which are of the form $[i,i]$. Denote the equivalence relation corresponding to~$\part$ by $\sim_\part$. Let~$\lesssim^\part$ be the preorder on $\Lambda$ generated by $\lesssim$ and $\sim_\part$.  Then $(\Lambda,\lesssim^\part)$ is $\mathbb{Z}$-like. 

If all $\mathfrak{p}_\alpha\in\part$ are either equal to $[i,i]$ for some~$i$ or are strongly convex, then $\lesssim^\part = \lesssim \,\cup\, \part$, that is, 
$$
\left\{ (x,y)\in\Lambda \mid x\lesssim^\part y\right\} = \left\{ (x,y)\in\Lambda \mid x\lesssim y\right\}\cup \left\{(x,y)\in\Lambda \mid x\sim_\part y\right\}.
$$
\label{prop:Z-like2}
\end{enumerate}
\end{prop}
\begin{proof}
\ref{prop:Z-like1} That $\lesssim$ is reflexive and transitive is clear, and strong convexity of the subsets $\Lambda_m\sqcup \dots \sqcup \Lambda_n$ holds because $\mathbb{Z}$ is linearly ordered. Then $(\bigsqcup_n \Lambda_n,\lesssim)$ is $\mathbb{Z}$-like because every finite subset of $\bigsqcup_n \Lambda_n$ is contained in $\Lambda_m\sqcup \dots \sqcup \Lambda_n$ for some $m,n$.

\ref{prop:Z-like2} Let $\mathfrak{q}$ be a finite subset of $\Lambda$ and $\mathfrak{p}_1$, \dots, $\mathfrak{p}_n$ be the parts of $\part$ that are not of the form $[i,i]$. Then $\mathfrak{q}\cup\bigcup_{l=1}^n \mathfrak{p}_l$ is a finite subset of $\Lambda$ and so there a finite set $\Lambda'$ strongly convex in $(\Lambda, \lesssim)$ such that $\mathfrak{q}\cup\bigcup_{l=1}^n \mathfrak{p}_l\subseteq \Lambda'$. Suppose that $i\not\in\Lambda'$. Then either $i\lnsim j$ for all $j\in \Lambda'$ or $i\lnsim j$ for all $j\in\Lambda'$ because $(\Lambda,\lesssim)$ is $\mathbb{Z}$-like. Since $\mathfrak{p}_l\subset \Lambda'$ for all~$l\in\{1,\dots,n\}$, we also have that $i\not\sim_\part j$ for all $j\in \Lambda'$. Hence either $i\lnsim^\part j$ for all $j\in \Lambda'$ or $i\lnsim\part j$ for all $j\in\Lambda'$ and $\Lambda'$ is strongly convex. Since $\mathfrak{q}$ was arbitrary, if follows that $(\Lambda,\lesssim^\part)$ is $\mathbb{Z}$-like. 
 
Suppose that all parts of $\part$ not of the form $[i,i]$ are strongly convex in $(\Lambda,\lesssim)$. If $i\sim_\part j$ and $j\lesssim k$, then either $k \in \mathfrak{p}$ and $i \sim_\part k$, or $k \not\in \mathfrak{p}$ and $i\lesssim k$ because $j\lesssim k$ and $\mathfrak{p}$ is strongly convex. A similar argument holds if~$i\lesssim j$ and~$j\sim_\part k$. Therefore $\lesssim \,\cup\, \part$ is a transitive relation and is equal to the preorder it generates. 
\end{proof}

The following terms and notation will be used when Proposition~\ref{prop:Z-like} is applied.
\begin{defn}
\label{defn:finitary}
\begin{itemize}
\item A partition~$\part$ of $\Lambda$ is \emph{finitary} if the parts of~$\part$ are finite convex subsets and all but finitely many are the minimal intervals $[k,k]$ with $k\in\Lambda$.
\item Given a finite convex subset $\Lambda'\subset \Lambda$, the finitary partition
$$
\{\Lambda'\}\sqcup\left\{ [k,k] \subset \Lambda \mid k\in \Lambda\setminus\Lambda'\right\}
$$ 
will be denoted by $\part\Lambda'$.
\item The preordered set $(\Lambda,\lesssim^\part)$ defined in Proposition~\ref{prop:Z-like}\ref{prop:Z-like2} will be denoted by $\Lambda+\part$. When $\part$ is the partition $\part\Lambda'$, with $\Lambda'$ a finite convex subset of $\Lambda$, the preordered set will be simply written as $\Lambda+\Lambda'$.
\end{itemize}
\end{defn}
Note that, when $\Lambda'$ is a finite strongly convex subset of $\Lambda$, the preorder $\Lambda+\Lambda'$ agrees with that of $\Lambda$ except that all elements of $\Lambda'$ are equivalent.

Matrix operations and associated notation are defined next. The examples to be studied are matrix groups over a finite field, $\mathbb{F}_q$, but the definitions apply to any commutative ring~$R$. 
\begin{defn}
\label{defn:matrix_ops}
Fix a set $\Lambda$ and a commutative unital ring $R$; we write $R^*$ for the group of units of $R$.  A \defbold{$(\Lambda \times \Lambda)$-matrix} over $R$ is a tuple $(a_{ij})_{i,j \in \Lambda}$ such that $a_{ij} \in R$ for all $i,j \in \Lambda$.  Note that we can add any pair of $(\Lambda \times \Lambda)$-matrices entry by entry.  We define a partial operation of multiplication of $(\Lambda \times \Lambda)$-matrices: the product of $(a_{ij})$ and $(b_{ij})$ is given by $(c_{ij})$ where 
\begin{equation}
\label{eq:matrix_mult}
c_{ij} = \sum_{k \in \Lambda}a_{ik}b_{kj},
\end{equation}
subject to the requirement that the product is only defined if, for all $i,j \in \Lambda$, the sum defining $c_{ij}$ has only finitely many nonzero terms.  It then follows that matrix multiplication is associative, for the same reason as for finite-dimensional matrices.  Write $\mathrm{M}_{\Lambda}(R)$ for the set of $(\Lambda \times \Lambda)$-matrices equipped with the operation of addition (under which $\mathrm{M}_{\Lambda}(R)$ is an abelian group) and the partial operation of multiplication.  When $\Lambda$ is finite $\mathrm{M}_{\Lambda}(R)$ is a ring; we write $\mathrm{GL}_{\Lambda}(R)$ for the group of units of this ring.

Suppose for the rest of this definition that $\Lambda$ is a $\bZ$-like partially ordered set.  The $(\Lambda \times \Lambda)$-matrix $(a_{ij})$ is \defbold{nonsingular} if there is a finite convex $\Lambda'\subset\Lambda$ such that, for every finite convex subset $\Lambda''\supset \Lambda'$, the $(\Lambda'' \times \Lambda'')$-submatrix $(a_{ij})_{i,j\in\Lambda''}$ is invertible over $R$. The matrix is \defbold{$\Lambda$-diagonal} if $a_{ij} = 0$ whenever $i \not\sim j$, and \defbold{scalar} if $a_{ij} = 0$ whenever $i \ne j$ and $a_{ii}$ is constant as $i$ ranges over $\Lambda$.  We remark that any matrix can be multiplied on either side by a $\Lambda$-diagonal matrix, and moreover that every matrix commutes with the scalar matrices.  In particular, the matrix $I = (\delta_{ij})$ is an identity element for multiplication on $\mathrm{M}_{\Lambda}(R)$. Also observe that, for each $k\in\Lambda$, the restriction of the $\Lambda$-diagonal matrices to the $[k,k]\times[k,k]$ `block' $\{(i,j)\mid i,j\sim k\}$ produces a ring isomorphic to $\mathrm{M}_{[k,k]}(R)$. 
\end{defn}

Subsets of $\mathrm{M}_{\Lambda}(R)$ which will be seen to be groups under the matrix multiplication are defined and named next. 
\begin{defn}
\label{defn:sets_of_matrices}
The group of nonsingular $\Lambda$-diagonal matrices is denoted by $\Delta_{\Lambda}(R)$ and its subgroup of nonsingular scalar matrices by $\Z_{\Lambda}(R)$. 

Given a $\bZ$-like partially ordered set $\Lambda$, a $(\Lambda \times \Lambda)$-matrix $(a_{ij})$ is: \defbold{upper triangular} if $a_{ij} = 0$ whenever $i \not\lesssim j$; \defbold{strictly upper triangular} if $a_{ij} = 0$ whenever $i \gtrsim j$; and \defbold{almost upper triangular} if $a_{ij} = 0$ for all but finitely many pairs $i,j$ such that $i \not\lesssim j$. Write $\mathrm{U}_{\Lambda}(R)$ for the set of nonsingular upper triangular $(\Lambda \times \Lambda)$-matrices; $\mathrm{U}^*_{\Lambda}(R)$ for the set of matrices $(\delta_{ij}+a_{ij})$ with $(a_{ij})$ strictly upper triangular; and $\AU_{\Lambda}(R)$ for the set of nonsingular almost upper triangular $(\Lambda \times \Lambda)$-matrices. 
 \end{defn}

\begin{rmk}
\label{rem:upper_triangular}
Suppose that the $(\Lambda \times \Lambda)$-matrix $(a_{ij})$ is almost upper triangular. Then there is a finite, strongly convex~$\Lambda'\subseteq \Lambda$ such that $a_{ij}=0$ unless $i, j \in \Lambda'$ or $i\lesssim j$ and, recalling the notation of Definition~\ref{defn:finitary}, $(a_{ij})$ belongs to $\mathrm{U}_{\Lambda+\Lambda'}(R)$. Conversely, if $\part$ is a {finitary} partition of $\Lambda$, then $\mathrm{U}_{\Lambda+\part}(R)\leq \AU_{\Lambda}(R)$. It follows, therefore, that, if $\Lambda$ has more than one element, then
\begin{equation*}
\label{eq:direct_limit}
\mathrm{U}_{\Lambda}(R) \subset \AU_{\Lambda}(R) = \bigcup\left\{  \mathrm{U}_{\Lambda+\part}(R) \mid \part \mbox{ a finitary partition of }\Lambda\right\}.
\end{equation*}
Denote the set of finitary partitions of~$\Lambda$ by~$\mc{F}$ and order~$\mc{F}$ by reverse refinement, that is, $\part_1\leq\part_2$ if~$\part_1$ is a refinement of~$\part_2$. Then $\AU_{\Lambda}(R)$ is the direct limit 
\begin{equation}
\label{eq:direct_limit2}
\mathrm{U}_{\Lambda}(R) \subset \AU_{\Lambda}(R) = \lim_{\part\in\mc{F}} \mathrm{U}_{\Lambda+\part}(R).
\end{equation}  
\end{rmk}

\section{Intermediate Results}

To begin our study of these matrix groups, we show that they are closed under multiplication, and that products of such matrices can easily be understood in terms of products of finite submatrices.

\begin{lem}
\label{lem:upper_triangular:product}
Let $R$ be a commutative unital ring, let $\Lambda$ be a $\bZ$-like partially ordered set and suppose that $(a_{ij})$ and $(b_{ij})$ belong to $\mathrm{U}_\Lambda(R)$. 
\begin{enumerate}[label=(\roman*)]
\item The product $(c_{ij}) = (a_{ij})(b_{ij})$ is well-defined.
\label{lem:upper_triangular:product1}
\item The matrix $(c_{ij})$ belongs to $\mathrm{U}_\Lambda(R)$.
\label{lem:upper_triangular:product2}
\item For every convex subset, $\Lambda'$, of $\Lambda$, we have
\[
(c_{ij})_{i,j \in \Lambda'} = (a_{ij})_{i,j \in \Lambda'}(b_{ij})_{i,j \in \Lambda'}.
\]
\label{lem:upper_triangular:product3}
\end{enumerate}
\end{lem}

\begin{proof}
Suppose $i,j,k \in \Lambda$ are such that $a_{ik}b_{kj} \neq 0$. Then $a_{ik}, b_{kj} \neq 0$ and so $i \lesssim k$ and $k \lesssim j$. Hence $i\lesssim j$, because~$\lesssim$ is transitive, and $k \in [i,j]$. Since $\Lambda$ is $\mathbb{Z}$-like, $[i,j]$ is finite, and we see that the sum
\begin{equation}
\label{eq:sum}
c_{ij} := \sum_{k \in \Lambda}a_{ik}b_{kj}
\end{equation}
is well-defined for all $i,j \in \Lambda$ and is zero unless $i \lesssim j$. This proves~\ref{lem:upper_triangular:product1} and~\ref{lem:upper_triangular:product2}.

What is more, let $\Lambda'$ be a convex subset of $\Lambda$ and $i,j \in \Lambda'$. Then all nonzero terms of the sum on the right hand side of (\ref{eq:sum}) arise from $k\in [i,j]\subseteq \Lambda'$, proving~\ref{lem:upper_triangular:product3}.
\end{proof}

Invertible elements of $\mathrm{M}_{\Lambda}(R)$ have a simple characterization. 
\begin{lem}\label{lem:upper_triangular}
Let $R$ be a commutative unital ring, let $\lesssim$ be a preorder on a set $\Lambda$ such that $(\Lambda,\lesssim)$ is $\mathbb{Z}$-like, and let $(a_{ij})$ be an upper triangular $(\Lambda\times\Lambda)$-matrix over~$R$.  Then $(a_{ij})$ is nonsingular if and only if the finite matrix $(a_{ij})_{i,j\in[k,k]}$ is nonsingular for all $k\in \Lambda$. 

If $(a_{ij})$ is nonsingular, then it is invertible in $\mathrm{M}_{\Lambda}(R)$. Hence $\mathrm{U}_{\Lambda}(R)$ is a group.
\end{lem}
\begin{proof}
Let us note first that the conclusions are clear when $(a_{ij})$ is a $\Lambda$-diagonal matrix.  Indeed, when $(a_{ij})$ is $\Lambda$-diagonal its inverse matrix (if it exists) is the $\Lambda$-diagonal matrix~$(d_{ij})$ such that $(d_{ij})_{i,j\in[k,k]} := ((a_{ij})_{i,j\in[k,k]})^{-1}$ for each~$k$.

For the general case, if $(a_{ij})_{i,j\in[k,k]}$ is not invertible for some~$k\in\Lambda$, then $(a_{ij})_{i,j\in\Lambda'}$ is not invertible for every $\Lambda'\supseteq [k,k]$ and $(a_{ij})$ is not nonsingular. Suppose then that $(a_{ij})_{i,j\in[k,k]}$ is invertible for all~$k$. Then the $\Lambda$-diagonal matrix $(d_{ij})$ such that $(d_{ij})_{i,j\in[k,k]} := ((a_{ij})_{i,j\in[k,k]})^{-1}$ is invertible and the product $(d_{ij})(a_{ij}) =: (c_{ij})$ is defined, by Proposition~\ref{lem:upper_triangular:product}. Calculation shows that $(c_{ij}) = (\delta_{ij}-s_{ij})$ where $(s_{ij})$ is strictly upper-triangular. For any given $i,j \in \Lambda'$, the $(i,j)$-entry of $(s_{ij})^n$ is zero for all but finitely many $n \in \bN$: specifically, the $(i,j)$-entry is given by the sum
\[
\sum_{k_1,\dots,k_n \in \Lambda'}s_{ik_1}s_{k_1k_2} \dots s_{k_{n-1}k_n}s_{k_nj};
\]
for a nonzero term we must have $i \lnsim k_1 \lnsim \dots \lnsim k_n \lnsim j$, which in particular implies that $n < |[i,j]|$.  Thus the infinite sum
\begin{equation}
\label{eq:power}
(r_{ij}) := \sum_{n \ge 0} (s_{ij})^n
\end{equation}
is a well-defined matrix; it is clearly also upper triangular. This matrix satisfies $(r_{ij})(c_{ij}) = (\delta_{ij}) = (c_{ij})(r_{ij})$, that is, $(r_{ij}) = (\delta_{ij}-s_{ij})^{-1}$. Hence the product $(r_{ij})(d_{ij})$ is also a well-defined upper triangular matrix, and we have
\[
(r_{ij})(d_{ij})(a_{ij}) = (a_{ij})(r_{ij})(d_{ij}) = (\delta_{ij}).
\]
Thus $(r_{ij})(d_{ij})$ is an inverse for $(a_{ij})$. Hence $(a_{ij})$ is invertible and $\mathrm{U}_{\Lambda}(R)$ is a group.
\end{proof}

The argument of Lemma~\ref{lem:upper_triangular} may be taken further, as follows.
\begin{lem}
\label{lem:groups}
Let $R$ be a commutative unital ring and $\Lambda$ be a $\bZ$-like preordered set. Then $\mathrm{U}^*_{\Lambda}(R)$ and $\Delta_{\Lambda}(R)$ are subgroups of $\mathrm{U}_{\Lambda}(R)$, with $\mathrm{U}^*_{\Lambda}(R)$ normal, and  $\mathrm{U}_{\Lambda}(R)$ decomposes as
\[
\mathrm{U}_{\Lambda}(R) = \mathrm{U}^*_{\Lambda}(R) \rtimes  \Delta_{\Lambda}(R).
\]
Furthermore, 
$$
\Delta_{\Lambda}(R)\cong \prod_{[k,k],\, k\in\Lambda} \mathrm{GL}_{n(k)}(R)
$$ 
with $n(k) = |[k,k]|$.
\end{lem}
\begin{proof}
It was already observed at the beginning of the proof of Lemma~\ref{lem:upper_triangular} that $\Delta_{\Lambda}(R)$ is a subgroup of $\mathrm{U}_{\Lambda}(R)$. Furthermore, the calculation in that proof shows that each $(a_{ij})\in \mathrm{U}_{\Lambda}(R)$ satisfies $(a_{ij}) = (d_{ij})^{-1}(\delta_{ij} - s_{ij})$ and that every matrix $(\delta_{ij} - s_{ij})$ with $(s_{ij})$ strictly upper triangular is invertible with inverse given by~\eqref{eq:power}. Hence $\mathrm{U}^*_{\Lambda}(R)$ is a subgroup of $\mathrm{U}_{\Lambda}(R)$ too, and $\mathrm{U}_{\Lambda}(R) = \Delta_{\Lambda}(R)\mathrm{U}^*_{\Lambda}(R)$. 

That $\mathrm{U}^*_{\Lambda}(R)\cap \Delta_{\Lambda}(R)$ is trivial is clear, and so for the first claim it remains only to show that $\mathrm{U}^*_{\Lambda}(R)$ is a normal subgroup. For this, consider $(a_{ij}) \in \mathrm{U}_{\Lambda}(R)$ and $(\delta_{ij} - s_{ij})\in A \in \mathrm{U}^*_{\Lambda}(R)$. Lemma~\ref{lem:upper_triangular:product}\ref{lem:upper_triangular:product3} implies that, for every $k\in\Lambda$, 
$$
((a_{ij})(\delta_{ij} - s_{ij}))_{i,j\in[k,k]} = (a_{ij})_{i,j\in[k,k]},
$$  
and it follows, again by Lemma~\ref{lem:upper_triangular:product}\ref{lem:upper_triangular:product3}, that 
$$
((a_{ij})(\delta_{ij} - s_{ij})(a_{ij})^{-1})_{i,j\in[k,k]} = (\delta_{ij})_{i,j\in[k,k]}.
$$ 
Hence $\mathrm{U}^*_{\Lambda}(R)$ is normal.

For the second claim, observe that the minimal convex set $[k,k]$ is finite for every $k\in\Lambda$ and that, if an invertible matrix $(a_{ij})_{i,j\in[k,k]}$ is chosen for every such $[k,k]$, then the $(\Lambda\times\Lambda)$-matrix $(b_{ij})$ with
$$
b_{ij} = \begin{cases}
a_{ij}, & \text{ if }i,j\in[k,k]\text{ for some }k\in\Lambda\\
0, & \text{ if } i\lnsim j\text{ or }j\lnsim i
\end{cases}
$$
belongs to $\Delta_{\Lambda}(R)$. Lemma~\ref{lem:upper_triangular} shows that every element of $\Delta_{\Lambda}(R)$ has this form and the claimed isomorphism follows. 
\end{proof}

The following observation, derived from Lemma~\ref{lem:upper_triangular:product}\ref{lem:upper_triangular:product3} and Lemma~\ref{lem:upper_triangular}, will be useful later.

\begin{rmk}
\label{rem:upper_triangular:product}
For every convex subset $\Lambda'$, the restriction of the identity matrix in $\mathrm{M}_{\Lambda}(R)$ to $\Lambda'$ is the identity matrix in $\mathrm{M}_{\Lambda'}(R)$. Hence  Lemma~\ref{lem:upper_triangular:product}\ref{lem:upper_triangular:product3} implies that the restriction map $\theta_{\Lambda'} : \mathrm{U}_\Lambda(R) \to \mathrm{U}_{\Lambda'}(R)$ defined by $\theta_{\Lambda'}((a_{ij})_{i,j \in \Lambda}) = (a_{ij})_{i,j \in \Lambda'}$  is a group homomorphism. If $(a_{ij})_{i,j \in \Lambda'}\in\mathrm{GL}_{\Lambda'}(R)$, then the matrix in $\mathrm{M}_{\Lambda}(R)$ which agrees with $a_{ij}$ when $i,j\in\Lambda'$ and equals $\delta_{ij}$ otherwise belongs to $\mathrm{U}_{\Lambda+\Lambda'}(R)$. Hence $\theta_{\Lambda'} : \mathrm{U}_{\Lambda+\Lambda'}(R) \to \mathrm{GL}_{\Lambda'}(R)$ is a surjection.

\end{rmk}

Let $\mc{C}$ be the set of finite convex subsets of $\Lambda$. Then combining the homomorphisms, $\theta_{\Lambda'}$, of Remark~\ref{rem:upper_triangular:product} and noting that $\mathrm{U}_{\Lambda'}(R)\leq \mathrm{GL}_{\Lambda'}(R)$ yields an injective group homomorphism
\begin{equation}
\label{eq:product_map1}
\theta : \mathrm{U}_{\Lambda}(R) \rightarrow \prod_{\Lambda' \in \mc{C}} \mathrm{GL}_{\Lambda'}(R); \; (a_{ij}) \mapsto \left((a_{ij})_{i,j \in \Lambda'}\right)_{\Lambda'\in\mc{C}}.
\end{equation}

To finish this section, we remark that the centre of $\AU_{\Lambda}(R)$ is now easily obtained from the centre of the general linear group.

\begin{lem}\label{lem:centre}
Let $R$ be a commutative unital ring and let $\Lambda$ be a $\bZ$-like partially ordered set.  Then $\Z(\AU_{\Lambda}(R)) = \Z_{\Lambda}(R)$.
\end{lem}

\begin{proof}
It is clear that $\Z_{\Lambda}(R)$ is a central subgroup of $\AU_{\Lambda}(R)$.  Conversely, given a central element $(a_{ij})$ of $\AU_{\Lambda}(R)$, then we see from Lemma~\ref{lem:groups} that $(a_{ij})_{i,j \in \Lambda'}$ is central in $\mathrm{GL}_{\Lambda'}(R)$ for all finite convex subsets $\Lambda'$ of $\Lambda$ such that $(a_{ij}) \in \mathrm{U}_{\Lambda+\Lambda'}(R)$. Hence $(a_{ij})_{i,j \in \Lambda'}$ is a scalar diagonal matrix.  Given the freedom of choice of $\Lambda'$, we conclude that $(a_{ij})_{i,j \in \Lambda}$ is a scalar diagonal matrix, that is, $(a_{ij}) \in \Z_{\Lambda}(R)$.
\end{proof}
\section{Topology}

Suppose that $R$ is a topological ring. Equip $\mathrm{M}_{\Lambda}(R)$ with the product topology and $\mathrm{U}_{\Lambda}(R)$ with the subspace topology. Direct products of topological groups, such as in Equation~\eqref{eq:product_map1}, are also equipped with the product topology.

\begin{lem}
\label{lem:uppertriangular_topology}
Let $R$ be a commutative unital topological ring and let $\Lambda$ be a $\bZ$-like preordered set. Then $\mathrm{U}_{\Lambda}(R)$ is a topological group.  Moreover, the map $\theta$ defined in (\ref{eq:product_map1}) is a closed topological embedding.
\end{lem}

\begin{proof}
By definition of the product topology, the map 
$$
\pi(a,b) = ab^{-1} : \mathrm{U}_{\Lambda}(R)\times\mathrm{U}_{\Lambda}(R) \to \mathrm{U}_{\Lambda}(R)
$$ 
is continuous if $\theta_{ij}\circ \pi$ is continuous for all $(i,j)\in \Lambda^2$, where $\theta_{ij}(a) = a_{ij}$. For this, it suffices to show that $\theta_{\Lambda'}\circ \pi$ is continuous for every finite convex subset $\Lambda'$. 

Remark~\ref{rem:upper_triangular:product} implies that $\theta_{\Lambda'}\circ \pi = \pi\circ (\theta_{\Lambda'}\times\theta_{\Lambda'})$, and $\pi\circ (\theta_{\Lambda'}\times\theta_{\Lambda'})$ is continuous if the restriction of $\pi$ to $\theta_{\Lambda'}(\mathrm{U}_{\Lambda}(R))$, which is equal to~$\mathrm{U}_{\Lambda'}(R)$, is continuous. Hence continuity of $\theta_{\Lambda'}\circ \pi$ amounts to $\mathrm{U}_{\Lambda'}(R)$ being a topological group for every finite convex set $\Lambda'$.
Since $\mathrm{GL}_{\Lambda'}(R)$ is a topological group when equipped with the subspace topology for the product topology on $\mathrm{M}_{\Lambda'}(R)$, and since $\mathrm{U}_{\Lambda'}(R)$ is defined by a set of equations and hence is a closed subgroup of $\mathrm{GL}_{\Lambda'}(R)$, it follows that $\mathrm{U}_{\Lambda'}(R)$ is indeed a topological group.

The map $\theta$ is continuous because each homomorphism $\theta_{\Lambda'} : (a_{ij}) \mapsto (a_{ij})_{i,j\in\Lambda'}$ is continuous, and is a homeomorphism onto its range because each of the coordinate maps $\theta_{ij}$ (which determine the product topology on $\mathrm{M}_{\Lambda}(R)$) factors through $\theta_{\Lambda'}$ if $\Lambda'$ contains~$i$ and~$j$. The image is closed because it can be specified by a set of equations on the entries.
\end{proof}

If $R$ is finite and discrete, then $\mathrm{GL}_{\Lambda'}(R)$ is a finite discrete group for each finite convex $\Lambda'$ and we have the following immediate consequence.

\begin{cor}
\label{cor:uppertriangular_topology}
Let $R$ be a finite commutative unital ring equipped with the discrete topology, and let $\Lambda$ be a $\bZ$-like preordered set.  Then $\mathrm{U}_{\Lambda}(R)$ is a profinite group.
\endproof
\end{cor}

In the case that~$R$ is discrete, the description of $\AU_{\Lambda}(R)$ as the direct limit of profinite groups $\mathrm{U}_{\Lambda+\part}(R)$ given in Equation~\eqref{eq:direct_limit2} may be used to extend the topology on $\mathrm{U}_{\Lambda}(R)$ to $\AU_{\Lambda}(R)$. 
Recall that Definition~\ref{defn:finitary} identifies each $\Lambda'\in\mc{C}$ with the finitary partition $\part{\Lambda'}$ and denote the subset of~$\mc{F}$ consisting of all such partitions by $\part\mc{C}$. The assumption that $\Lambda$ is $\mathbb{Z}$-like then implies that $\part\mc{C}$ is a cofinal subset of $(\mc{F},\leq)$ and Equation~\eqref{eq:direct_limit2} becomes, in the notation of Definition~\ref{defn:finitary},
\begin{equation*}
\AU_{\Lambda}(R) = \lim_{\Lambda'\in\mc{C}} \mathrm{U}_{\Lambda+\Lambda'}(R).
\end{equation*}
The topologies for $\mathrm{U}_{\Lambda+\Lambda'}(R)$ as $\Lambda'$ ranges over~$\mc{C}$ are consistent with one another in the following sense: given $\Lambda_1,\Lambda_2 \in \mc{C}$, the intersection $\mathrm{U}_{\Lambda+{\Lambda_1}}(R) \cap \mathrm{U}_{\Lambda+{\Lambda_2}}(R)$, being determined by a condition on finitely many entries in $\mathrm{U}_{\Lambda+{\Lambda_i}}(R)$, is open in both $\mathrm{U}_{\Lambda+{\Lambda_1}}(R)$ and $\mathrm{U}_{\Lambda+{\Lambda_2}}(R)$ and carries the subspace topology in both.  It follows that there is a unique group topology for $\AU_{\Lambda}(R)$ such that the embedding of $\mathrm{U}_{\Lambda+\Lambda'}$ into $\AU_{\Lambda}(R)$ is continuous and open for all $\Lambda'\in\mc{C}$.  If $R$ is finite, we see that the topology of $\AU_{\Lambda}(R)$ is locally profinite, that is, $\AU_{\Lambda}(R)$ is a totally disconnected, locally compact group.

\begin{rmk}
\label{rem:baseof1}
For each $\Lambda'\in\mc{C}$, the subgroup 
$$
\mathrm{U}^{\Lambda'}_{\Lambda}(R) := \left\{ (a_{ij})\in \mathrm{U}_{\Lambda}(R) \mid a_{ij} = \delta_{ij} \text{ if }i,j\in\Lambda'\right\}
$$
is open in $\AU_{\Lambda}(R)$; in fact, these subgroups form a base of identity neighbourhoods in $\AU_{\Lambda}(R)$ because every finite subset of $\Lambda$ is contained in an element of~$\mc{C}$.
\end{rmk}
\section{Normal subgroups}
\label{sec:normal_subgroups}

We now consider the closed normal subgroups of $\AU_{\Lambda}(R)$. Just as for finite-dimensional matrix groups, there is a natural family of `principal congruence subgroups' arising from the ring structure of $R$.  Specifically, if $I$ is a proper ideal of $R$, then the map $(a_{ij}) \mapsto (a_{ij} + I)$ induces a group homomorphism from $\AU_{\Lambda}(R)$ to $\AU_{\Lambda}(R/I)$.  Provided that $I$ is nonzero, the kernel of this map is a proper nontrivial closed normal subgroup.  The question of how other closed normal subgroups of $\AU_{\Lambda}(R)$ relate to the principal congruence subgroups appears to be difficult. 

Since we are interested in topologically simple groups, we will focus on the case that $R$ is a field. We recall a well-known simplicity result, which may be found in, for example, \cite[\S\S103--105]{Dickson}.
\begin{lem}\label{lem:finite_dim:simple}
Let $F$ be a field and let $n \in \bN$ such that $n \ge 2$; in the case $n=2$, assume $|F| > 3$.  Then every proper normal subgroup of $\mathrm{SL}_n(F)$ is central, and every noncentral normal subgroup of $\mathrm{GL}_n(F)$ contains $\mathrm{SL}_n(F)$.
\end{lem}

The analogous result for $\AU_{\Lambda}(F)$ may now be deduced.
\begin{thm}\label{thm:simple}
Let $F$ be a finite discrete field and let $\Lambda$ be an infinite $\bZ$-like partially ordered set.  Then $\Z_{\Lambda}(F)$ is the unique largest proper closed normal subgroup of $\AU_{\Lambda}(F)$.  In particular, $\AU_{\Lambda}(F)/\Z_{\Lambda}(F)$ is topologically simple.
\end{thm}

\begin{proof}
It is clear that $\Z_{\Lambda}(F)$ is a proper closed normal subgroup of $\AU_{\Lambda}(F)$.  To show that it is the unique largest one, it suffices to consider a closed normal subgroup $N$ of $\AU_{\Lambda}(F)$ that is not contained in $\Z_{\Lambda}(F)$ and show that $N = \AU_{\Lambda}(F)$.

Given such $N$ and $(a_{ij}) \in N$ which is not a scalar matrix, there is a finite strongly convex $\Lambda'\subset\Lambda$ with $|\Lambda'| \ge 3$ and such that $(a_{ij}) \in \mathrm{U}_{\Lambda+\Lambda'}(F)$ and $(a_{ij})_{i,j \in \Lambda'}$ is not scalar. Suppose that $\Lambda'$ is any such finite strongly convex subset of $\Lambda$. Then, applying the homomorphism $\theta_{\Lambda'}$ given in Remark~\ref{rem:upper_triangular:product}, we see that 
$$
N_{\Lambda'} := \theta_{\Lambda'}(\mathrm{U}_{\Lambda+\Lambda'}(F) \cap N)
$$ 
is noncentral and normal in $\mathrm{GL}_{\Lambda'}(F)$. Hence $\mathrm{SL}_{\Lambda'}(F) \le N_{\Lambda'}$, by Lemma~\ref{lem:finite_dim:simple}. Since $F$ is a field, each matrix $(f_{ij})$ in $\mathrm{GL}_{\Lambda'}(F)$ is a submatrix of a matrix $(g_{ij})$ in $\mathrm{U}_{\Lambda'}(F)$ with $g_{ij} = f_{ij}$ for $i,j\in\Lambda'$ and $g_{ij}=0$ when $i\ne j$ and $i$ or $j$ is not in $\Lambda'$. Moreover, $(g_{ij})$ may be chosen with $(g_{ij})\in \mathrm{SL}_{\Lambda''}(F)$ for some convex $\Lambda''$ strictly containing $\Lambda'$. Since $N_{\Lambda''} \geq \mathrm{SL}_{\Lambda''}(F)$ by the previous argument and $(f_{ij}) = \theta_{\Lambda'}(g_{ij})$, it follows that $N_{\Lambda'}=\mathrm{GL}_{\Lambda'}(F)$.

Consider now an arbitrary element $(c_{ij})$ of $\AU_{\Lambda}(F)$ and suppose that $\Lambda'$ is sufficiently large that $(c_{ij}) \in \mathrm{U}_{\Lambda+\Lambda'}(F)$.  Then $(c_{ij})_{i,j \in \Lambda'}$ is an element of $\mathrm{GL}_{\Lambda'}(F)$ and there is $(b_{ij}) \in N$ such that $(c_{ij})_{i,j \in \Lambda'} = \theta_{\Lambda'}(b_{ij})$.  Since this holds for all sufficiently large $\Lambda'$, $(c_{ij})$ is approximated by elements of $N$ in the topology of entrywise convergence and, since $N$ is closed, it follows that $(c_{ij}) \in N$.  This completes the proof that $N = \AU_{\Lambda}(F)$.

In particular, any nontrivial closed normal subgroup of $\AU_{\Lambda}(F)/\Z_{\Lambda}(F)$ has preimage equal to $\AU_{\Lambda}(F)$ and $\AU_{\Lambda}(F)/\Z_{\Lambda}(F)$ is topologically simple.
\end{proof}

We conclude this section by showing that $\AU_{\Lambda}(F)/\Z_{\Lambda}(F)$ is not simple. To this end, observe that $\AU_{\Lambda}(F)$ acts on the vector space
$$
\mathrm{L}(F) := \left\{ (x_j)\in F^\Lambda \mid \text{ there is } k\in\Lambda\text{ with }x_l=0 \text{ for all }k\lesssim l\right\}
$$
by matrix multiplication because $\sum_{j\in\Lambda} a_{ij}x_j$ is a finite sum for all $(a_{ij})\in \AU_{\Lambda}(F)$ and $(x_j)\in \mathrm{L}(F)$. Define 
\begin{equation}
\label{eq:dense_subgroup}
\mathcal{AU}_{\Lambda}(F) = \left\{ (\delta_{ij} + b_{ij})\in \AU_{\lambda} \mid (x_j)\mapsto \sum_{j\in\Lambda} b_{ij}x_j \text{ has finite rank }\right\}.
\end{equation}
Then $\mathcal{AU}_{\Lambda}(F)$ is closed under multiplication because the sum and product of finite rank operators have finite rank, and is closed under the inverse because every element is equal to the identity on a finite-codimensional subspace of $\mathrm{L}(F)$ and hence so is its inverse. Therefore $\mathcal{AU}_{\Lambda}(F)$ is a subgroup of $\AU_{\Lambda}(F)$. That it is a normal subgroup follows because the rank of an operator does not change under conjugation. 

The subgroup $\mathcal{AU}_{\Lambda}(F)\cap \Delta_\Lambda(F)$ of $\Lambda$-diagonal matrices in $\mathcal{AU}_{\Lambda}(F)$ has infinite index in  $\AU_{\Lambda}(F)\cap \Delta_\Lambda(F)$ and so $\mathcal{AU}_{\Lambda}(F)$ and $\mathcal{AU}_{\Lambda}(F)\Z_{\Lambda}(F)$ are proper subgroups of $\AU_{\Lambda}(F)$. Since $\Z_{\Lambda}(F)$ is contained in the closure of $\mathcal{AU}_{\Lambda}(F)$, Theorem~\ref{thm:simple} implies the following.
\begin{prop}
\label{prop:dense_normal}
The group $\mathcal{AU}_{\Lambda}(F)$ defined in~\eqref{eq:dense_subgroup} is a proper dense normal subgroup of $\AU_{\Lambda}(F)$.
\end{prop}

\section{Local structure of infinite matrix groups}
\label{sec:local_uncountable}

`Local structure' of a totally disconnected, locally compact group $G$ refers to properties of compact open subgroups of $G$ which are preserved under commensurability. We investigate next how well this invariant distinguishes totally disconnected, locally compact infinite matrix groups.

\subsection{An uncountable number of non-isomorphic groups $\AU_{\Lambda}(F)$}
\label{sec:uncountable}

For each partition $\qart$ of $\mathbb{Z}$ ($\mathbb{N}$ or $-\mathbb{N}$) into finite intervals, define a $\mathbb{Z}$-like preorder by 
$$
m\lesssim_\qart n \text{ if } m<n \text{ or }m,n\in\mathfrak{p}\text{ for some } \mathfrak{p}\in\qart
$$
and denote $(\mathbb{Z},\lesssim_\qart)$ by $[\qart]$. Furthermore, let 
$$
\sharp(\qart) = \left\{ |\mathfrak{p}|\in \mathbb{N} \mid \mathfrak{p}\in \qart\right\}.
$$

A \defbold{local isomorphism} of topological groups $G$ and $H$ is an open embedding $\phi: U \rightarrow H$, where $U$ is an open neighbourhood of the identity in $G$ and $\phi$ is compatible with the group operations (as far as they are defined on $U$).  We say two topological groups are \defbold{locally isomorphic} if there is a local isomorphism between them.
\begin{prop}
\label{prop:local_structure}
Let $\qart$, $\qart_1$ and $\qart_2$ be partitions of $\mathbb{Z}$ into finite intervals. 
\begin{enumerate}
\item \label{prop:local_structure2}
There is a continuous surjective homomorphism $\mathrm{U}_{[\qart]}(F) \to \mathrm{PGL}_n(F)$ with $n>1$ (or $n>2$ if $|F|=3$) if and only if $n\in \sharp(\qart)$.
\item \label{prop:local_structure3}
If $\AU_{[\qart_1]}(F)$ is locally isomorphic to $\AU_{[\qart_2]}(F)$, then $\sharp(\qart_1)\Delta \sharp(\qart_2)$ is finite. 
\end{enumerate}
\end{prop} 
\begin{proof}
\eqref{prop:local_structure2}
It is shown in Lemma~\ref{lem:groups} that $\mathrm{U}_{[\qart]}(F) = \mathrm{U}^*_{[\qart]}(F)\rtimes \Delta_{[\qart]}(F)$ and that $\Delta_{[\qart]}(F)\cong \prod_{\mathfrak{p}\in\qart} \mathrm{GL}_{|\mathfrak{p}|}(F)$, and then $\mathrm{PGL}_{|\mathfrak{p}|}(F)$ is the quotient of $\mathrm{GL}_{|\mathfrak{p}|}(F)$ by its centre.  Hence there is a homomorphism $\mathrm{U}_{[\qart]}(F) \to \mathrm{PGL}_n(F)$ for any $n = |\mathfrak{p}|$ with $\mathfrak{p}\in\qart$.

For the converse, consider a continuous homomorphism $\phi : \mathrm{U}_{[\qart]}(F)\to \mathrm{PGL}_n(F)$. Then $\phi( \mathrm{U}^*_{[\qart]}(F) )$ is a normal subgroup of $\mathrm{PGL}_n(F)$, because $\mathrm{U}^*_{[\qart]}(F)$ is normal in $\mathrm{U}_{[\qart]}(F)$, and then since $\mathrm{PGL}_n(F)$ is simple, the only possibilities are that $\phi( \mathrm{U}^*_{[\qart]}(F) )$ is trivial or equal to $\mathrm{PGL}_n(F)$. Since the commutator subgroups $\mathrm{U}^*_{[\qart]}(F)^{(n)}$ in the descending series for $\mathrm{U}^*_{[\qart]}(F)$ converge to the trivial subgroup, whereas $[\mathrm{PGL}_n(F),\mathrm{PGL}_n(F)] = \mathrm{PGL}_n(F)$, it must be that $\phi( \mathrm{U}^*_{[\qart]}(F) )$ is trivial.  Since $\mathrm{U}_{[\qart]}(F) = \mathrm{U}^*_{[\qart]}(F)\rtimes \Delta_{[\qart]}(F)$, it follows that $\phi$ restricts to a surjective homomorphism from $\Delta_{[\qart]}(F)$.  From the structure of $\Delta_{[\qart]}(F)$, and since $\mathrm{PGL}_n(F)$ is only isomorphic to a quotient of $\mathrm{GL}_m(F)$ if $m = n$, we see that the surjective homomorphism from $\Delta_{[\qart]}(F)$ to $\mathrm{PGL}_n(F)$ must restrict to a surjective homomorphism from $\mathrm{GL}_{\mathfrak{p}}(F)$ to $\mathrm{PGL}_n(F)$, for some $\mathfrak{p} \in \qart$ such that $|\mathfrak{p}| = n$.  In particular, $n\in \sharp(\qart)$.

\eqref{prop:local_structure3}
Let $\psi : U \to \AU_{[\qart_2]}(F)$ be a local isomorphism from $\AU_{[\qart_1]}(F)$ to $\AU_{[\qart_2]}(F)$.  Since $\mathrm{U}_{[\qart_1]}(F)$ is an open profinite subgroup of $\AU_{[\qart_1]}(F)$, we can take $U$ to be an open normal subgroup of $\mathrm{U}_{[\qart_1]}(F)$; since $\mathrm{U}_{[\qart_2]}(F)$ is open in $\AU_{[\qart_2]}(F)$, by restricting to a smaller domain we may assume that $\psi(U) \le \mathrm{U}_{[\qart_2]}(F)$.  Hence, by Remark~\ref{rem:baseof1}, there is a finite interval $[m,n]\subset \mathbb{Z}$ such that 
$$
\mathrm{U}^{[m,n]}_{[\qart_2]}(F)\leq \psi(U) \leq \mathrm{U}_{[\qart_2]}(F).
$$
Each $\mathfrak{p}\in\qart_2$ yields a homomorphism $\phi_{\mathfrak{p}} : \mathrm{U}_{[\qart_2]}(F)\to \mathrm{PGL}_{|\mathfrak{p}|}(F)$ whose restriction to $\mathrm{U}^{[m,n]}_{[\qart_2]}(F)$, and hence to $\psi(U)$, is surjective provided that~$\mathfrak{p}$ is disjoint from $[m,n]$.  Supposing that~$\mathfrak{p}$ is disjoint from $[m,n]$ and $|\mathfrak{p}|>2$, we can then extend $\phi_{\mathfrak{p}}\circ\psi$ to a surjective homomorphism
$$
\widetilde{\phi_{\mathfrak{p}}}: \mathrm{U}^*_{[\qart_1]}(F)U \to  \mathrm{PGL}_{|\mathfrak{p}|}(F),
$$
where $\widetilde{\phi_{\mathfrak{p}}}(\mathrm{U}^*_{[\qart_1]}(F))$ is trivial; we then see that $\mathrm{U}^*_{[\qart_1]}(F)U \cap \Delta_{[\qart_1]}(F)$ is an open normal subgroup of $\Delta_{[\qart_1]}(F)$ that surjects onto $\mathrm{PGL}_{|\mathfrak{p}|}(F)$, and hence $|\mathfrak{p}|\in \sharp(\qart_1)$.  Thus $\sharp(\qart_1) \setminus \sharp(\qart_2)$ is finite; similarly, $\sharp(\qart_2) \setminus \sharp(\qart_1)$ is finite, so $\sharp(\qart_1)\Delta \sharp(\qart_2)$ is finite.
\end{proof}

The number of subsets of $\mathbb{N}$ modulo the equivalence relation of finite symmetric difference is uncountable. Proposition~\ref{prop:local_structure} therefore implies 
\begin{cor}
\label{cor:local_structure}
For each prime power~$q$, there are uncountably many local isomorphism classes of topologically simple totally disconnected locally compact groups of the form $\AU_{\Lambda}(\mathbb{F}_q)$ with $(\Lambda,\lesssim)$ a $\mathbb{Z}$-like preorder.
\end{cor}

The construction just given for $(\mathbb{Z},\lesssim_{\qart})$ also produces uncountably many local isomorphism classes of topologically simple groups $\AU_{\Lambda}(\mathbb{F}_q)$ where $\Lambda$ is of the form $(\mathbb{N},\lesssim_{\qart})$ or ${(-\mathbb{N},\lesssim_{\qart})}$. Regarding isomorphism classes rather than local isomorphism classes, it is likely that $\sharp(\qart)$ is not a sufficiently fine invariant to distinguish all non-isomorphisms between pairs of such groups, because there is no obvious isomorphism between $\AU_{[\qart_1]}(F)$ and $\AU_{[\qart_2]}(F)$ if all intervals appearing in $\qart_1$ and $\qart_2$ have the same lengths but the lengths appear in a different order. For example, let 
\begin{align*}
\qart_1 &= \left\{ [20n,20n+9]\mid n\in\mathbb{Z}\right\}\cup \left\{\{m\}\mid m\in[20n+10,20n+19],\ n\in\mathbb{Z}\right\}\text{ and }\\
\qart_2 &= \left\{ [100n,100n+9]\mid n\in\mathbb{Z}\right\}\cup \left\{\{m\}\mid m\in[100n+10,100n+99],\ n\in\mathbb{Z}\right\}.
\end{align*}
Are $\AU_{[\qart_1]}(F)$ and $\AU_{[\qart_2]}(F)$ isomorphic?

\subsection{Locally normal subgroups}
\label{sec:locnorm}

Suppose that $G$ is a totally disconnected, locally compact group and that $U$ is a compact open subgroup of~$G$. A subgroup $H\leq U$ is \defbold{locally normal} if the normaliser of $H$ is open. This concept is defined in~\cite{CapReidWiI,CapReidWiII} where the lattice of commensurability classes of locally normal subgroups is studied. We recall further concepts from~\cite{CapReidWiI,CapReidWiII} before investigating locally normal subgroups in $\AU_{\Lambda}(\mathbb{F}_q)$.

The \defbold{quasicentre} $\mathrm{QZ}(G)$ of a locally compact group~$G$ is the set of elements whose centraliser is open. Compactly generated topologically simple groups have trivial quasicentre by~\cite[Theorem~A]{CapReidWiII} while, on the other hand, the topologically simple groups constructed in~\cite{WSimp} have dense quasicentre. It turns out that the local structure of the non-compactly generated topologically simple groups constructed here is closer to that of compactly generated groups in this respect.
\begin{prop}
\label{prop:trivialQC}
Let $(\Lambda,\lesssim)$ be a $\mathbb{Z}$-like preorder. Then $\mathrm{QZ}(\AU_{\Lambda}(\mathbb{F}_q)) = \Z_{\Lambda}(\mathbb{F}_q)$.
\end{prop}
\begin{proof}
Consider $(x_{ij})_{i,j\in\Lambda}\in \AU_{\Lambda}(\mathbb{F}_q)$ and suppose that $(x_{ij})_{i,j\in\Lambda}$ centralises the open neighbourhood~$\mathrm{U}^{\Lambda'}_{\Lambda}(\mathbb{F}_q)$ of the identity for some finite convex set $\Lambda'\subset\Lambda$, see Remark~\ref{rem:baseof1} for the definition of these subgroups. Increasing $\Lambda'$ if necessary, it may be supposed that $(x_{ij})_{i,j\in\Lambda}$ belongs to $\mathrm{U}_{\Lambda+\Lambda'}(\mathbb{F}_q)$. Then $\theta_{\Lambda''}((x_{ij})_{i,j\in\Lambda})$ centralises $\theta_{\Lambda''}(\mathrm{U}^{\Lambda'}_{\Lambda}(\mathbb{F}_q))$, see Remark~\ref{rem:upper_triangular:product} for the definition of $\theta_{\Lambda''}$, for every convex $\Lambda''\supset\Lambda'$. Thus $\theta_{\Lambda''}((x_{ij})_{i,j\in\Lambda})\in \mathrm{U}_{\Lambda''}(\mathbb{F}_q)$ and lies in the centre of
$$
\left\{ (a_{ij})_{i,j}\in \mathrm{U}_{\Lambda''}(\mathbb{F}_q) \mid a_{ij}=\delta_{ij}\text{ if }i,j\in\Lambda'\right\},
$$
which is a group of upper triangular $(\Lambda''\times\Lambda'')$-matrices. The elements of $\mathrm{U}_{\Lambda''}(\mathbb{F}_q)$ are block matrices with blocks indexed by the minimal convex sets $[k,k]$, $k\in\Lambda''$, while $\theta_{\Lambda''}((x_{ij})_{i,j\in\Lambda})$ is a block matrix with blocks indexed by $\{\Lambda'\}\sqcup \left\{ [k,k]\mid k \not\in \Lambda'\right\}$. Provided that $k\lnsim l$ and at least one of $k$ and $l$ is not in $\Lambda'$, $\theta_{\Lambda''}(\mathrm{U}^{\Lambda'}_{\Lambda}(\mathbb{F}_q))$ contains all matrices which differ from the identity matrix only in the $([k,k]\times[l,l])$-block. Since $\theta_{\Lambda''}((x_{ij})_{i,j\in\Lambda})$ commutes with all such matrices, choosing $\Lambda''$ sufficiently larger than $\Lambda'$ forces $\theta_{\Lambda''}((x_{ij})_{i,j\in\Lambda})$ to be a scalar matrix. Since this holds for all sufficiently large $\Lambda''$, it follows that $\theta_{\Lambda''}((x_{ij})_{i,j\in\Lambda})$ is scalar.
\end{proof}

\begin{cor}
\label{cor:trivialQC}
$\mathrm{QZ}(\AU_{\Lambda}(\mathbb{F}_q)/\Z_{\Lambda}(\mathbb{F}_q)) = \{1\}$.
\end{cor}
\begin{proof}
Suppose that $x \Z_{\Lambda}(\mathbb{F}_q)\in \AU_{\Lambda}(\mathbb{F}_q)/\Z_{\Lambda}(\mathbb{F}_q)$ centralises the open subgroup~$U$ and let $V$ be an open subgroup of $\AU_{\Lambda}(\mathbb{F}_q)$ such that $V / \Z_{\Lambda}(\mathbb{F}_q)\leq U$. Then the map $v\mapsto [x,v]$ is a continuous homomorphism $V\to \Z_{\Lambda}(\mathbb{F}_q)$ which has an open kernel because $\Z_{\Lambda}(\mathbb{F}_q)$ is discrete. Hence $x\in \Z_{\Lambda}(\mathbb{F}_q)$, by Proposition~\ref{prop:trivialQC}.
\end{proof}

The notion of quasicentre is relativised to subgroups, $H$, in~\cite{CapReidWiI}: the \defbold{quasicentraliser of $H$ in $G$} is 
$$
\mathbf{QC}_G(H) = \left\{ g\in G \mid g \text{ centralises an open subgroup of }H\right\}.
$$
The role of this concept is that the lattice of locally normal subgroups in~$G$ has a sublattice consisting of the centralisers of locally normal subgroups. This sublattice may be shown to be a Boolean lattice if every locally normal subgroup of~$G$ satisfies the following condition.
\begin{defn}
\label{defn:Cstable}
Let $G$ be a totally disconnected, locally compact groups and $U$ be a compact open subgroup of~$G$. The subgroup~ $H$ of $G$ is said to be \defbold{C-stable} in $G$ if
$$
\mathbf{QC}_G(H) \cap \mathbf{QC}_G(\mathbf{C}_G(H)) \text{ is commensurable with } \{1_G\}.
$$
\end{defn}
It may be shown that this condition is satisfied independently of the choice of compact open subgroup~$U$. The group~$G$ is called \defbold{locally C-stable} if all locally normal subgroups of $G$ are C-stable in $G$.

Theorem~5.3 in~\cite{CapReidWiII} shows every compactly generated topologically simple totally disconnected locally compact group is locally C-stable. On the other hand, all topologically simple groups~$G$ constructed in~\cite{WSimp} are not locally C-stable because, $\mathrm{QZ}(G)$ being dense, compact open subgroups in~$G$ are not C-stable. The groups constructed here also fail to be locally C-stable but the reason is less obvious.

\begin{examp}
\label{examp:NN}
Let $G = \AU_{\mathbb{N}}(\mathbb{F}_q)$. Then the compact open subgroup $\mathrm{U}_{\mathbb{N}}(\mathbb{F}_q)$ consists of all upper triangular matrices over $\mathbb{F}_q$ which have non-zero entries on the diagonal. The subgroup
$$
H = \left\{ (a_{ij})_{i,j\in\mathbb{N}} \mid a_{ii}=1 \text{ and } a_{ij}=0 \text{ if }i>j\text{ or } j>i>1\right\},
$$
that is, matrices whose only off-diagonal non-zero entries are in the first row, is normal in $\mathrm{U}_{\mathbb{N}}(\mathbb{F}_q)$ and hence is locally normal. This subgroup is abelian and so $\mathbf{QC}_G(H)\geq \Z_{\mathbb{N}}(\mathbb{F}_q)H = \mathbf{C}_G(H)$. Hence~$H$ is not C-stable. The same claim holds for the subgroup $H\Z_{\mathbb{N}}(\mathbb{F}_q)/\Z_{\mathbb{N}}(\mathbb{F}_q)$ in the topologically simple group $\AU_{\mathbb{N}}(\mathbb{F}_q)/\Z_{\mathbb{N}}(\mathbb{F}_q)$.

On the other hand, for each $n\in\mathbb{N}$ the subgroup
$$
L_n = \left\{ (a_{ij})_{i,j\in\mathbb{N}}\mid a_{ii}=1 \text{ and } a_{ij}=0 \text{ if } i\ne j\text{ and } i>j-n\right\},
$$
that is, matrices whose only off-diagonal non-zero entries are above the $n^{\mathrm{th}}$ superdiagonal, is normal in $\mathrm{U}_{\mathbb{N}}(\mathbb{F}_q)$ and hence is locally normal. These subgroups satisfy $\mathbf{QC}_G(L_n) = \Z_{\mathbb{N}}(\mathbb{F}_q) = \mathbf{C}_G(L_n)$. Hence $L_n\Z_{\mathbb{N}}(\mathbb{F}_q)/\Z_{\mathbb{N}}(\mathbb{F}_q)$, $n\in\mathbb{N}$, are C-stable subgroups of $\AU_{\mathbb{N}}(\mathbb{F}_q)/\Z_{\mathbb{N}}(\mathbb{F}_q)$.
\end{examp}

\begin{examp}
\label{examp:ZZ}
Let $G = \AU_{\mathbb{Z}}(\mathbb{F}_q)$. Then $\mathrm{U}_{\mathbb{Z}}(\mathbb{F}_q)$ is a compact open subgroup of~$G$. the locally C-stable when index set is $\mathbb{Z}$
For each $k\in\mathbb{Z}$ let 
$$
B_k = \left\{ (a_{ij})\in \mathrm{U}_{\mathbb{Z}}(\mathbb{F}_q) \mid a_{ij} -\delta_{ij} = 0 \text{ unless } i< k \text{ and }j>k\right\}.
$$
Then, for $(a_{ij}),\,(b_{ij})\in B_k$ we have, by \eqref{eq:matrix_mult},
$$
(a_{ij})(b_{ij}) = \bigl(\sum_{l\in\mathbb{Z}} a_{il}b_{lj}\bigr) = (-\delta_{ij} + a_{ij}+ b_{ij})
$$
because the only way in which both $a_{il}$ and $b_{lj}$ can be non-zero is if $l$ is equal to at least one of $i$ and $j$. Hence $B_k$ is an abelian subgroup of $\mathrm{U}_{\mathbb{Z}}(\mathbb{F}_q)$. A similar calculations shows that, unless $i<k$ and $j>k$, if $(a_{ij})\in\mathrm{U}_{\mathbb{Z}}(\mathbb{F}_q)$ and $(b_{ij}) = B_k$, then $(a_{ij})(b_{ij})= (a_{ij}) = (b_{ij})(a_{ij})$. Hence $B_k$ is a normal subgroup of $\mathrm{U}_{\mathbb{Z}}(\mathbb{F}_q)$ and is therefore locally normal. Then $B_k$ is not  is not C-stable, and neither is $B_k\Z_{\mathbb{N}}(\mathbb{F}_q)/\Z_{\mathbb{N}}(\mathbb{F}_q)$ in the topologically simple group $\AU_{\mathbb{N}}(\mathbb{F}_q)/\Z_{\mathbb{N}}(\mathbb{F}_q)$.

The subgroups defined analogously to the subgroups $L_n$ in Example~\ref{examp:NN} are C-stable however.
\end{examp}

\end{document}